\documentclass[reqno]{amsart}

\usepackage{amsmath,amssymb,latexsym,amsthm,times}

\usepackage[unicode]{hyperref}

\theoremstyle{theorem}
\newtheorem{theorem}{Theorem}
\newtheorem{lemma}{Lemma}
\newtheorem*{conjecture}{Conjecture}
 
\theoremstyle{definition}

\begin{document}

\title[Integer Ratios of Consecutive Alternating Power Sums]{Integer Ratios of Consecutive Alternating\\ Power Sums}

\author{Ioulia N. Baoulina}

\address{Department of Mathematics, Moscow State Pedagogical University, Krasnoprudnaya str. 14, Moscow 107140, Russia}
\email{jbaulina@mail.ru}

\date{}

\maketitle

\thispagestyle{empty}

\begin{abstract}
We give a characterization of all pairs $(k,n)$ of positive integers for which the ratio $(1^k-2^k+3^k-\dots+(-1)^{n+1} n^k)/(1^k-2^k+3^k-\dots+(-1)^{n}(n-1)^k)$ of two consecutive alternating power sums is an integer.
\end{abstract}

\section{Introduction}

Power sums have attracted the attention of mathematicians for centuries. Although many of their properties are well known, there are still open questions. One of the unsolved problems is the following: find all positive integers $k$ and $n>1$ such that the ratio
$$
\frac{1^k+2^k+\dots+n^k}{1^k+2^k+\dots+(n-1)^k}
$$
is an integer. Put $S_k(n)=\sum_{j=1}^{n-1} j^k$. Clearly, $S_k(3)/S_k(2)=2^k+1$ is always an integer. Further, it is not difficult to show that $S_k(4)/S_k(3)$ is an integer if and only if $k=1$ or $3$ (see \cite{baoulina-moree} or \cite{moree2}). In 2011, Kellner~\cite{kellner} formulated the following conjecture.

\begin{conjecture}[Kellner, 2011]
Let $k\ge 1$ and $n\ge 3$ be integers. Then the ratio\linebreak ${S_k(n+1)/S_k(n)}$ is a integer if and only if $n=3$ and $k=1$ or $3$.
\end{conjecture}

This conjecture remains open, although various partial results are known (see \cite{baoulina-moree}). In particular, it has been proved in \cite{baoulina-moree} that if $n>3$ and $S_k(n+1)/S_k(n)$ is an integer, then either $S_k(n+1)/S_k(n)=2$ or $S_k(n+1)/S_k(n)>1500$. In the latter case,\linebreak $(S_k(n+1)/S_k(n))-1$ must be a product of irregular primes.

In this note, we consider a similar problem for alternating power sums. Namely, we are looking for positive integers $k$ and $n>1$ such that the ratio
$$
\frac{1^k-2^k+3^k-\dots+(-1)^{n+1} n^k}{1^k-2^k+3^k-\dots+(-1)^{n}(n-1)^k}
$$
is an integer. Put $A_k(n)=(-1)^n\sum_{j=1}^{n-1} (-1)^{j+1} j^k$. With this notation,
$$
\frac{1^k-2^k+3^k-\dots+(-1)^{n+1} n^k}{1^k-2^k+3^k-\dots+(-1)^{n}(n-1)^k}=-\frac{A_k(n+1)}{A_k(n)},
$$
and so the problem above can be reformulated as follows: find all positive integers $k$ and $n>1$ such that the ratio $A_k(n+1)/A_k(n)$ is an integer. It is readily seen that
\begin{align*}
A_k(n)&=(n-1)^k-(n-2)^k+(n-3)^k-\dots+(-1)^{n+2}\cdot 1^k\\
&\ge\begin{cases}
(n-1)/2 & \text{if $n$ is odd,}\\
n/2 & \text{if $n$ is even,}
\end{cases}
\end{align*}
and so $A_k(n)>0$ for $n>1$. In view of this, it is more convenient to work with $A_k(n)$.

Unlike Kellner's conjecture, the problem for alternating power sums is not so difficult, and we are able to give a complete solution. Our main result is the following theorem.

\begin{theorem}
\label{main}
Let $k\ge 1$ and $n>1$ be integers. Then the ratio $A_k(n+1)/A_k(n)$ is an integer if and only if at least one of the following conditions holds:

\textup{(a)} $n=2$;

\textup{(b)} $k=1$ and $n$ is even;

\textup{(c)} $k=1$ or $2$ and $n=3$.

\end{theorem}

\section{Congruences for power sums}

In order to prove Theorem~\ref{main}, we need two auxiliary results.

\begin{lemma}
\label{L1}
Let $k>1$ and $n\ge 1$ be integers. Then
$$
2S_k(n)\equiv\begin{cases}
-2\sum_{p\mid n,\, (p-1)\mid k}\frac np \pmod{2n} & \text{if $k$ is even,}\\
-kn\sum_{p\mid n,\, (p-1)\mid (k-1)}\frac np \pmod{n^2} & \text{if $k$ is odd.}
\end{cases}
$$
\end{lemma}

The ``even" part of Lemma~\ref{L1} is due to von Staudt~\cite{staudt}. A simple proof of this result has been given by Moree~\cite{moree1}. The ``odd" part of Lemma~\ref{L1} is due to Me\v strovi\'c~\cite{mestrovic}.

For a real number $r$, let $\lfloor r\rfloor$ denote the greatest integer less than or equal to $r$.

\begin{lemma}
\label{L2}
Let $k>1$ and $n>1$ be integers. Then
$$
A_k(n)\equiv\begin{cases}
0\pmod{n(n-1)/2} & \text{if $k$ is even and $n$ is odd,}\\
0\pmod{n-1} & \text{if $k$ and $n$ are even,}\\
\lfloor n/2\rfloor^2\pmod{2\lfloor n/2\rfloor^2} & \text{if $k$ is odd.}
\end{cases}
$$
\end{lemma}

\begin{proof}
First assume that $n\ge 3$ is odd. We have
\begin{align*}
A_k(n)&=\sum_{\substack{j=1\\ 2\mid j}}^{n-1} j^k-\sum_{\substack{j=1\\ 2\nmid j}}^{n-1} j^k=2\sum_{\substack{j=1\\ 2\mid j}}^{n-1} j^k-\sum_{j=1}^{n-1} j^k=2^{k+1}\sum_{i=1}^{(n-1)/2} i^k-S_k(n)\\
&=2^{k+1}S_k((n-1)/2)+2(n-1)^k-S_k(n-1)-(n-1)^k.
\end{align*}
Hence
\begin{equation}
\label{eq1}
A_k(n)\equiv 2^{k+1}S_k((n-1)/2)-S_k(n-1)\pmod{(n-1)^2/2}.
\end{equation}
Applying Lemma~\ref{L1}, we obtain, for even $k$,
\begin{align*}
A_k(n)&\equiv -2^{k+1}\sum_{\substack{p\mid\frac{n-1}2\\ (p-1)\mid k}} \frac{n-1}{2p}+\sum_{\substack{p\mid (n-1)\\ (p-1)\mid k}} \frac{n-1}p\\
&\equiv (1-2^k)\sum_{\substack{p\mid (n-1)\\ (p-1)\mid k\\ p>2}} \frac{n-1}p\equiv (n-1)\sum_{\substack{p\mid (n-1)\\ (p-1)\mid k\\ p>2}} \frac{1-2^k}p\pmod{(n-1)/2}.
\end{align*}
By Fermat's little theorem, $p\mid(1-2^k)$ for every prime $p>2$ with ${(p-1)\mid k}$. Thus $\sum_{p\mid (n-1),\, (p-1)\mid k,\, p>2} (1-2^k)/p$ is an integer, and so  $A_k(n)\equiv 0\pmod{(n-1)/2}$. Moreover, if $k$ is even, then 
$$
A_k(n)=\sum_{j=1}^{n-1} (-1)^j j^k=\sum_{j=1}^{n-1}(-1)^{n-j} (n-j)^k\equiv  -A_k(n)\pmod{n}.
$$
Hence $2A_k(n)\equiv 0\pmod{n}$, and thus $A_k(n)\equiv 0\pmod{n}$. Since $(n-1)/2$ and $n$ are coprime, we conclude that $A_k(n)\equiv 0\pmod{n(n-1)/2}$.

Similarly, if $k\ge 3$ is odd, then \eqref{eq1} and Lemma~\ref{L1} imply
\begin{align*}
A_k(n)&\equiv -2^k k\cdot\frac{n-1}2\sum_{\substack{p\mid\frac{n-1}2\\ (p-1)\mid (k-1)}} \frac{n-1}{2p}+k\cdot\frac{n-1}2\sum_{\substack{p\mid (n-1)\\ (p-1)\mid (k-1)}} \frac{n-1}p\\
& \equiv k\cdot\frac{(n-1)^2}4+(1-2^{k-1})k\cdot\frac{n-1}2\sum_{\substack{p\mid (n-1)\\ (p-1)\mid (k-1)\\ p>2}} \frac{n-1}p\\
&\equiv \frac{(n-1)^2}4+k\cdot\frac{(n-1)^2}2\sum_{\substack{p\mid (n-1)\\ (p-1)\mid (k-1)\\ p>2}} \frac{1-2^{k-1}}p\pmod{(n-1)^2/2}.
\end{align*}
Hence, again by Fermat's little theorem, $A_k(n)\equiv (n-1)^2/4\pmod{(n-1)^2/2}$.

Now assume that $n$ is even. Since the case $n=2$ is trivial, we can assume that $n\ge 4$. Then $n-1\ge 3$ and $n+1$ are odd. By what has been proved above,
\begin{align*}
A_k(n-1)&\equiv 0\pmod{n-1} && \text{if $k$ is even,}\\
A_k(n+1)&\equiv n^2/4\pmod{n^2/2} && \text{if $k$ is odd.}
\end{align*}
Recalling that $k>1$, we find that 
\begin{align*}
A_k(n)&=(n-1)^k-A_k(n-1)\equiv 0\pmod{n-1} && \text{if $k$ is even,}\\
A_k(n)&=n^k-A_k(n+1)\equiv -n^2/4\equiv n^2/4\pmod{n^2/2} && \text{if $k$ is odd.}
\end{align*}
This completes the proof.
\end{proof}

\section{Proof of the main result}

First we note that $A_k(3)/A_k(2)=2^k-1$ is an integer for every $k$. Further, it is easy to see that $A_1(n)=\lfloor n/2\rfloor$. Thus, for $n>1$,
$$
\frac{A_1(n+1)}{A_1(n)}=\begin{cases}
1 & \text{if $n$ is even,}\\
(n+1)/(n-1) & \text{if $n$ is odd.}
\end{cases}
$$
Hence $A_1(n+1)/A_1(n)$ is an integer if and only if either $n$ is even or $n=3$.

Now suppose that $k>1$, $n>2$ and the ratio $A_k(n+1)/A_k(n)$ is an integer. Since $A_k(n+1)=n^k-A_k(n)$, we must have $cA_k(n)=n^k$ for some positive integer~$c$. If both $k$ and $n$ are even, then, by Lemma~\ref{L2}, $n-1$ divides $n^k$, which is impossible for $n>2$. If $n$ is odd, then, again by Lemma~\ref{L2}, $(n-1)/2$ divides $n^k$. This implies that $n=3$ and $c(2^k-1)=3^k$, and hence $2^k-1=3^m$ for some positive integer $m$. If $k\ge 3$, then we must have $3^m\equiv -1\pmod{8}$, which is impossible. Hence $k=2$ and $m=1$.

It remains to examine the case when $k\ge 3$ is odd and $n\ge 4$ is even. In this case Lemma~\ref{L2} yields
$A_k(n)\equiv n^2/4\pmod{n^2/2}$. Hence
$$
\frac{A_k(n)}{n^2/4}\equiv 1\pmod{2}.
$$
Writing $cA_k(n)=n^k$ in the form
$$
c\cdot \frac{A_k(n)}{n^2/4}=2^k\left(\frac n2\right)^{k-2},
$$
we conclude that $2^k$ divides $c$, and so $c\ge 2^k$. Furthermore, since $n\ge 4$ is even,
\begin{align*}
A_k(n)&=1+\sum_{j=1}^{(n-2)/2} \bigl((2j+1)^k-(2j)^k\bigr)=1+\sum_{j=1}^{(n-2)/2} j^k\left(\Bigl(2+\frac 1j\Bigr)^k-2^k\right)\\
&>1+\sum_{j=1}^{(n-2)/2} j^k\left(\Bigl(1+\frac 1j\Bigr)^k-1^k\right)=1+\sum_{j=1}^{(n-2)/2} \bigl((j+1)^k-j^k\bigr)\\
&=\left(\frac n2\right)^k.
\end{align*}
Consequently, $n^k=cA_k(n)>2^k\cdot(n/2)^k=n^k$, and this is the desired contradiction.

\end{document}